\documentclass[a4paper,11pt]{amsart}
\usepackage{standalone}
\usepackage{amsfonts}
\usepackage{amsthm}
\usepackage{amssymb}
\usepackage{amsmath}
\usepackage{theoremref}
\usepackage{enumerate}
\usepackage{bbm}
\usepackage{bm}
\usepackage{pgfplots}
\pgfplotsset{compat=1.18}
\usepgfplotslibrary{fillbetween}
\usetikzlibrary{intersections}
\usepackage{mathtools}
\usetikzlibrary{patterns}

\usepackage{a4wide}

\numberwithin{equation}{section}

\newcommand{\A}{\mathcal{A}}

\newcommand{\N}{\mathcal{N}}

\newcommand{\T}{\mathbb{T}}
\newcommand{\To}{\mathcal{T}}

\newcommand{\conj}[1]{\overline{#1}}
\newcommand{\D}{\mathbb{D}}

\newcommand{\R}{\mathbb{R}}

\newcommand{\cD}{\conj{\mathbb{D}}}
\newcommand{\ip}[2]{\big\langle #1, #2 \big\rangle}

\newcommand{\BMOA}{\textbf{BMOA}}
\newcommand{\Mult}[1]{\textbf{Mult}(#1)}
\newcommand{\hil}{\mathcal{H}}

\newcommand{\hb}{\mathcal{H}(b)}

\newtheorem{mainthm}{Theorem}

\newtheorem{thm}{Theorem}[section]
\newtheorem*{thm*}{Theorem}
\newtheorem{lem}[thm]{Lemma}

\newtheorem{cor}[thm]{Corollary}
\newtheorem*{cor*}{Corollary}
\newtheorem{prop}[thm]{Proposition}
\theoremstyle{definition}

\theoremstyle{definition}

\newtheorem*{defn*}{Definition}

\newtheorem{claim*}{Claim}

\begin{document}


\baselineskip=17pt


\title[Universal multipliers for Sub-Hardy Hilbert spaces]{Universal multipliers for Sub-Hardy Hilbert spaces}

\author[B. Malman]{Bartosz Malman}
\address{Division of Mathematics and Physics \\
        Mälardalen University \\
		721 23 Västerås, Sweden}
\email{bartosz.malman@mdu.se} 

\author[D. Seco]{Daniel Seco}
\address{Universidad de la Laguna e IMAULL,  Avenida Astrof\'isico Francisco S\'anchez, s/n, Departamento de An\'alisis Matem\'atico \\ 38206 San Crist\'obal de La Laguna, Santa Cruz de Tenerife,  Spain} \email{dsecofor@ull.edu.es}

\date{}

\begin{abstract} To every non-extreme point $b$ of the unit ball of $\hil^\infty$ of the unit disk there corresponds a Pythagorean mate, a bounded outer function $a$ satisfying the equation $|a|^2 + |b|^2 = 1$ on the boundary of the disk. We study universal, i.e., simultaneous multipliers for families of de Branges-Rovnyak spaces $\hb$, and develop a general framework for this purpose. Our main results include a new proof of the Davis-McCarthy universal multiplier theorem for the class of all non-extreme spaces $\hb$, a characterization of the Lipschitz classes as the universal multipliers for spaces $\hb$ for which the quotient $b/a$ is contained in a Hardy space, and a similar characterization of the Gevrey classes as the universal multipliers for spaces $\hb$ for which $b/a$ is contained in a Privalov class.
\end{abstract}

\subjclass[2020]{Primary 30H45; Secondary 47B32}

\keywords{de Branges-Rovnyak spaces, multipliers}

\maketitle

\section{Introduction}

\subsection{De Branges-Rovnyak spaces}

The spaces $\hb$ were introduced by de Branges and Rovnyak in \cite{debSSS1966} and named after them. They form a family of Hilbert spaces of analytic functions on the unit disk $\D := \{ z \in \mathbb{C} : |z| < 1\}$ contained in the Hardy space $\hil^2$ of square-summable coefficient power series in $\D$. The family is parametrized by symbols $b$ in the unit ball of $\hil^\infty$, the algebra of bounded analytic functions in $\D$, and a given symbol $b$ defines uniquely the Hilbert space $\hb$ of functions with reproducing kernel of the form \[ k_b(z,\lambda) = \frac{1 - \conj{b(\lambda)} b(z)}{1-\conj{\lambda}z}, \quad z, \lambda \in \D.\] 
In general, it is not particularly easy to understand what functions are members of $\hb$. Various ways of constructing the space appear in the original work of de Branges and Rovnyak in \cite{debSSS1966}, Sarason's short treatise in \cite{sarasonbook}, and the more recent two-volume set \cite{hbspaces1fricainmashreghi}, \cite{hbspaces2fricainmashreghi} by Fricain and Mashreghi. A variety of applications of spaces $\hb$ to operator theory and complex analysis are also treated in those works.

In this article, we study the multiplier algebras of spaces $\hb$:
\[ \Mult{\hb} = \{ m \in \hil^\infty : mf \in \hb \text{ whenever } f \in \hb \}\]
The algebra $\Mult{\hb}$ may very well be trivial. In the case that the symbol $b$ is an inner function, then Crofoot observed in \cite{crofoot1994multipliers} that the only multipliers of $\hb$ are the constant functions. As usual, a function is \textit{inner} if it is bounded in $\D$ and has boundary values on $\T = \partial \D = \{ z \in \mathbb{C} : |z| = 1\}$ of unit modulus almost everywhere. From works of Lotto and Sarason in \cite{lotto1991multiplicative} and \cite{lotto1993multipliers} we know that plenty of non-constant multipliers exist whenever $b$ is not inner. However, an explicit characterization of $\Mult{\hb}$ exists only in a few very special cases. See our discussion in Section \ref{S:OtherResultsSubsection} below for some examples of results of this type.

We will concern ourselves only with the \textit{non-extreme case}, namely, we are assuming that the symbol $b$ satisfies the condition \[ \int_\T \log (1-|b(\zeta)|^2) |d\zeta| > -\infty.\] Here $|d\zeta|$ denotes the arclength measure on $\T$. The above integral convergence is well known to be equivalent to $b$ being a non-extreme point of the unit ball of $\hil^\infty$. It is customary to introduce the outer function $a:\D \to \D$ satisfying the equation \[ |a(\zeta)|^2 + |b(\zeta)|^2 = 1\] for almost every $\zeta \in \T$, and $a(0) > 0$. Then $a$ is uniquely determined by $b$, and we say that $b$ and $a$ form a \textit{Pythagorean pair}. These pairs will play a lead role in our discussion.

One can define a non-extreme space in a particularly useful way as a set of solutions to the \textit{mate equation},
\begin{equation}
    \label{E:MateEquationIntro}
    \To_{\conj{b}}f = \To_{\conj{a}}f_+.
\end{equation} Here $\To_{\conj{b}}$ and $\To_{\conj{a}}$ are Toeplitz operators on $\hil^2$, that is, operators of multiplication by $\conj{b}$ and $\conj{a}$ respectively, followed by an orthogonal projection from $L^2(\T)$ to $\hil^2$. The space $\hb$ can be defined as the set of those $f \in \hil^2$ for which a solution $f_+ \in \hil^2$ exists to the mate equation. It can be shown that at most one solution $f_+$ exists in $\hil^2$, so that the \textit{mate} $f_+$ is well-defined, if it exists. The norm on $\hb$ is then $\|f\|_{\hb} = \sqrt{\|f\|_2^2 + \|f_+\|_2^2}$, where $\| \cdot \|_2$ is the usual $L^2(\T)$ norm. We note also that the non-extreme case is characterized by the containment of the set of analytic polynomials in $\Mult{\hb}$, and their density in $\hb$. Both of these results are due to Sarason.

\subsection{Universal multipliers and the Davis-McCarthy theorem}

One special feature of the spaces $\hb$ is that they are not rotationally invariant. Namely, if $f(z) \in \hb$, then it is not in general the case that $f(e^{i\theta} z) \in \hb$. This feature is shared by the multiplier algebra $\Mult{\hb}$, and membership of a function $f$ in $\hb$ or in $\Mult{\hb}$ often depends on local behaviour of $f$ near distinguished points on $\T$ (for instance, this is the case in examples studied in \cite{fricain2022smirnov}).

Consider, however, a family $\mathcal{F}$ of symbols $b$ which is rotationally invariant: $b(e^{i\theta}z)\in \mathcal{F}$ for every $b(z) \in \mathcal{F}$ and $\theta \in \R$. Then certainly it is to be expected that $\cap_{b \in \mathcal{F}} \Mult{\hb}$ is a space which is at least rotationally invariant, and one may therefore hope for an easier characterization of this intersection. Any function $m$ inside the intersection of the multiplier algebras corresponding to $\mathcal{F}$ may justly be called a \textit{universal multiplier} for $\mathcal{F}$. Our families of symbols $b$ will be defined by membership of the \textit{Pythagorean quotient} $b/a$ in a sufficiently nice space of analytic functions $X$:
\[ \mathcal{F}(X) := \{ b : b/a \in X\}.\] For instance, if $\N^+$ is the Smirnov class of quotients of bounded analytic functions in $\D$ with outer denominator, then $\mathcal{F}(\N^+)$ is readily seen to equal the family of all non-extreme symbols $b$. The universal multipliers in this case have been characterized by Davis and McCarthy in their deep work \cite{davis1991multipliers}. 
For $\alpha \in (0,1/2]$, let $\mathcal{G}_\alpha$ be the Gevrey class \begin{equation}
    \label{E:GevreyEq} \mathcal{G}_{\alpha} := \Big\{ f(z) = \sum_{n \geq 0} \widehat{f}(n) z^n : |\widehat{f}(n)| = O \big(\exp(-cn^\alpha)\big) \text{ for some } c > 0 \Big\}.
\end{equation}
Using McCarthy's earlier work on topologies of $\N^+$ in \cite{mccarthy1990common}, they proved the following result.

\begin{thm*}[\textbf{Davis-McCarthy}]
\thlabel{T:DavisMcCarthyTheorem}
We have
\[\mathcal{G}_{1/2} = \bigcap_{b \in \mathcal{F}(\N^+)} \Mult{\hb}.\]
\end{thm*}

Motivated by the above theorem, we will consider in this article smaller rotationally invariant families $\mathcal{F}(X)$, and characterize their corresponding universal multipliers.

\subsection{Main result}

Our first main result pertains to symbol families with Pythagorean quotients in $\hil^p$. The Hardy space $\hil^p$ consists of those analytic functions in $\D$ which satisfy the integral mean boundedness condition 
\begin{equation}
\label{E:HpNorm}
\|f\|_p^p := \sup_{r \in (0,1)} \int_\T |f(r\zeta)|^p |d\zeta| < \infty.
\end{equation}  We will use the full range $p \in (0, \infty)$. For $\alpha \in (0,1)$, the analytic Lipschitz class $\Lambda^a_\alpha$ is defined as the set of those functions $m$ analytic in $\D$ and continuous in $\cD := \D \cup \T$ which satisfy the Lipschitz-type modulus of continuity estimate 
\[ |m(z)-m(w)| \leq C_m |z-w|^\alpha, \quad z,w \in \D\] for some constant $C_m > 0$. For $\alpha = 1$, in order to present a unified statement of our theorem, we follow a usual convention: we define $\Lambda^a_1$ as the analytic Zygmund class consisting of functions $m$ analytic in $\D$, continuous in $\cD$, and satisfying 
\[ |m(e^{i(t+s)}) + m(e^{i(t-s)}) - 2m(e^{it})| \leq C_m|s|, \quad s,t \in \R. \] For $\alpha > 1$, we define $\Lambda^a_\alpha$ as the space of functions for which the appropriate derivative lies in one of the above introduced classes. Namely, if $n$ is the positive integer satisfying $n < \alpha \leq n+1$, then $\Lambda^a_\alpha$ is to consist of those functions for which the derivative $f^{(n)}$ lies in $\Lambda^a_{\alpha -n}$. With these definitions in place, we can state our first main theorem.

\begin{mainthm}
    \thlabel{T:MainTheorem1}
    For $p \in (0,\infty)$ we have
    \begin{equation*}
        \Lambda^a_{1/p} = \bigcap_{b \in \mathcal{F}(\hil^p)} \Mult{\hb}.
    \end{equation*}
\end{mainthm}

One immediate corollary is that \[ \A^\infty = \bigcap_{p \in (0,\infty)} \bigcap_{b \in \mathcal{F}(\hil^p)} \Mult{\hb},\] where $\A^\infty = \bigcap_{\alpha > 0} \Lambda^a_{\alpha}$ is the algebra of analytic functions in $\D$ with smooth extensions to $\T$. 

Our second result pertains to the Privalov classes $\N^q$. Recall that functions in $\N^+$ satisfy $\int_\T \log(1+|f|)|d\zeta| < \infty$. For $q > 1$, the class $\N^q$ consists of those functions $f \in \N^+$ for which we have 
\[ \int_\T \big(\log(1+|f|)\big)^q |d\zeta| < \infty. \] The classes $\N^q$ have been studied by Privalov in \cite{PrivalovBook}. 

Recall the definition of the Gevrey classes $\mathcal{G}_a$ in \eqref{E:GevreyEq}. Our second main result is the following.

\begin{mainthm}
\thlabel{T:MainTheorem2}
For $q \in (1, \infty)$ we have 
\[\mathcal{G}_{1/(1+q)} = \bigcap_{b \in \mathcal{F}(\N^q)} \Mult{\hb}.\]
\end{mainthm}

In the proof of our theorems we will use mainly the methods of functional analysis, and our approach focused around the mate equation in \eqref{E:MateEquationIntro} is much different from the one used by Davis and McCarthy in \cite{davis1991multipliers}. The main results, \thref{T:MainTheorem1} and \thref{T:MainTheorem2}, will be deduced as consequences of a more general statement in \thref{P:MainProp} below. Given a space $X$ satisfying certain properties, we will establish in that proposition a bijection between universal multipliers for $\mathcal{F}(X)$ and symbols of certain Hankel operators. In particular, our general result applies to $X = \N^+$, and so we will obtain a new proof of the Davis-McCarthy theorem, one which is independent from McCarthy's work on topologies of the Smirnov class from \cite{mccarthy1990common}. Our starting point will be the research of Lotto and Sarason from \cite{lotto1993multipliers} which deals with operator-theoretic characterization of multipliers of $\hb$ in terms of boundedness of compositions of Hankel operators. An exposition of their results is contained in Section \ref{S:LottoSarasonResearchSection} below. 

\subsection{Other multiplier results and comments} \label{S:OtherResultsSubsection}

\subsubsection{More on the Davis-McCarthy theorem and its variants} 

The original proof of Davis-McCarthy universal multiplier theorem in \cite{davis1991multipliers} is based on earlier results of McCarthy from \cite{mccarthy1990common} on the dual of $\N^+$. There is a natural metric on $\N^+$ (see Section \ref{S:HankelGevreySubsec} below), and the dual with respect to the induced metric topology has been found by Yanagihara in \cite{yanagihara1973multipliers} to equal the Gevrey class $\mathcal{G}_{1/2}$ in \eqref{E:GevreyEq}. The dual of $\N^+$ with respect to a different topology, the so-called Helson topology, is easily seen to equal the intersection of ranges of all co-analytic Toeplitz operators on $\hil^2$ (see \cite{mccarthy1990common} for details). McCarthy showed that the two mentioned topologies have the same duals, and concluded that $\mathcal{G}_{1/2}$ equals the intersection of ranges of all co-analytic Toeplitz operators. The relevance of this result to spaces $\hb$ is that any multiplier for $\hb$ must necessarily be of the form $m = \To_{\conj{a}}u$ for some $u \in \hil^2$ (see \cite{lotto1993multipliers}, for instance), from which it easily follows that a universal multiplier must be contained in the mentioned intersection. That is, it necessarily must be a member of $\mathcal{G}_{1/2}$, by McCarthy's theorem. 

Other authors proved variants of the Davis-McCarthy theorem by following their strategy outlined in the above paragraph. In \cite{mevstrovic2009note}, Me\v{s}trovi\'c and Pavi\'cevi\'c used earlier duality results of Stoll from \cite{stoll1977mean} to find the universal multipliers corresponding to the family of spaces $\hb$ for which the logarithm of the density of the Aleksandrov-Clark measure of $b$ is in $L^q(\T)$, $q > 1$. The condition to be a universal multiplier for this family is the same as in \thref{T:MainTheorem2}, but the two results are different because the families of $\hb$-symbols are different. 

Similarly, our proof of \thref{T:MainTheorem1} is based on \thref{L:DRSDualityTheorem} below, which is a result of Duren, Romberg and Shields, but in fact going back to Hardy and Littlewood, and which identifies the dual space of $\hil^q$, for $q \in (0,1)$, as a Lipschitz class. Note, however, that in spite of this similarity, the proof technique used in the present article is completely different from the one used in \cite{davis1991multipliers} and the related works. One advantage of our approach is that it allows us to compute the universal multipliers for families of symbols $b$ defined in terms of their modulus alone, instead of the less tractable logarithmic integrability of their Aleksandrov-Clark densities.

\subsubsection{Explicit characterization of multipliers for rational $b$}

Dealing with a single symbol $b$ is a different and more delicate question, but an explicit characterization of $\Mult{\hb}$ exists in the case that $b$ is a rational function. In the case that a rational $b$ is not inner, then it is well known that $\hb = M(\conj{a})$, where the latter space is the range of the coanalytic Toeplitz operator $\To_{\conj{a}}$. By a result of Fricain, Hartmann and Ross from \cite{fricain2019multipliers} we conclude that $\Mult{\hb} = \hb \cap \hil^\infty$ if $b$ is rational and not inner. They establish also a very concrete description of the members of $\hb$ for these rational $b$. Naturally, for a general symbol $b$, such a simple characterization of $\Mult{\hb}$ is not available.

\subsubsection{All bounded functions as multipliers} 
The case $\hil^\infty = \Mult{\hb}$ has been settled by Sarason in \cite[Theorem 3]{sarasondoubly}, who gives several conditions equivalent to this equality. One of the equivalent conditons is that $\hb = a\hil^2 = \{ af : f \in \hil^2\}$. Another is that $a$ and $b$ should form a \textit{Corona pair}, in the sense that 
\[ \inf_{z \in \D} (|a(z)| + |b(z)|) > 0\] and $|a|^2$ should be a so-called $A_2$-weight: namely, it should satisfy the estimate
\[ \Bigg(\int_I |a|^2 |d\zeta| \Bigg) \Bigg(\int_I |a|^{-2} |d\zeta| \Bigg) \leq C|I|^2,\] where $I$ is any arc of $\T$ and $C > 0$ is a constant. Davis and McCarthy in \cite{davis1991multipliers} found a similar characterization in terms of the density $w := (1-|b|^2)|1-b|^{-1}$ of the Aleksandrov-Clark measure of $b$. By their result, $\Mult{\hb} = \hil^\infty$ if and only if $w$ is an $A_2$-weight.

\subsection{Outline of the paper and the methods} We begin, in Section \ref{S:LottoSarasonResearchSection}, by presenting the results of Lotto and Sarason from \cite{lotto1993multipliers}. Then, in Section \ref{S:HankelSection}, we characterize the continuity of Hankel operators between relevant pairs of function spaces. We continue in Section \ref{S:PythagoreanSection}, by proving a stability result for Pythagorean factorizations of functions in $\N^+$. We make use of these results in Section \ref{S:MainTheoremProofSection}, where we prove our general universal multiplier criterion and deduce from it \thref{T:MainTheorem1} and \thref{T:MainTheorem2}, as well as the Davis-McCarthy theorem. Then we conclude in Section \ref{sectfin} with some suggestions for continuing research along this direction.

\section{Research of Lotto and Sarason} \label{S:LottoSarasonResearchSection}

The purpose of this section is to present the research of Lotto and Sarason from \cite{lotto1993multipliers}. The important consequences of the results from that work are stated in \thref{C:LottoSarasonCor1} and \thref{C:LottoSarasonCor2} below.

\subsection{Hankel operators} \label{S:HankelOperatorDefSubsect}

Let $P_+$ and $P_-$ be the standard projection operators
\begin{equation}\label{E:PplusDef}
    P_+ f(z) = \int_{\T} \frac{f(\zeta)}{1 - \conj{\zeta}z}|d\zeta|, \quad z \in \D, \, f \in L^1(\T),
\end{equation}
\begin{equation}\label{E:PminusDef}
    P_-f (z) = \int_{\T} \frac{f(\zeta)\zeta \conj{z} }{1-\zeta \conj{z}} |d\zeta|, \quad z \in \D, \, f \in L^1(\T).
\end{equation}

The function $P_+f(z)$ is analytic in $\D$. On the other hand, $P_- f(z)$ is conjugate analytic in $\D$, and it vanishes at $z = 0$. Through the usual identification of the Hardy space $\hil^2$ with a closed subspace of $L^2(\T)$, as well as similar identification of the orthogonal complement $L^2(\T) \ominus \hil^2 =\conj{z \hil^2} = \{ \conj{zf} : f \in \hil^2\}$, the operators $P_+$ and $P_-$ are the orthogonal projections from $L^2(\T)$ onto $\hil^2$ and $\conj{z \hil^2}$, respectively. The actions of $P_+$ and $P_-$ on a Fourier series $f(\zeta) = \sum_{n \in \mathbb{Z}} \widehat{f}(n)\zeta^n \in L^2(\T)$ are given by
\[ P_+ f(\zeta) = \sum_{n \geq 0} \widehat{f}(n)\zeta^n \] and \[ P_- f(\zeta) = \sum_{n < 0 } \widehat{f}(n)\zeta^n.\]
We highlight the easily established identity 
\begin{equation}
    \label{E:PplusPminusEq}
    \conj{P_- f} = P_+\conj{f} - \conj{\widehat{f}(0)}.
\end{equation}

Given a symbol $m \in L^2(\T)$, we will consider the \textit{Hankel operator} 
\begin{equation}\label{E:HankelOpDef}
H_m f := P_- m f, \quad f \in \hil^\infty.
\end{equation} Clearly $H_m f (z)$ is a conjugate analytic function in $\D$ which vanishes at $z = 0$. In particular, we have $H_m f \in \conj{z\hil^2}$, and the conjugate $\conj{H_m f}$ is a member of the space $\hil^2$.

The content of Nehari's theorem from \cite{nehari1957bounded} is that if $m \in \hil^2$ equals the projection $m = P_+ u$ for some $u \in L^\infty(\T)$, then the operator $H_{\conj{m}}$ acts boundedly from $\hil^2$ into $\conj{z\hil^2}$, and that this condition on $m$ is necessary for boundedness. The space of such symbols, namely \[ \BMOA = \big\{ m \in \hil^2 : m = P_+ u \text{  for some u } \in L^\infty(\T) \big\} \] is the space of analytic functions of \textit{bounded mean oscillation}. It is well known that $\hil^1$ and $\BMOA$ are dual to each other, in the sense that every bounded linear functional $\ell$ on $\hil^1$ can be identified with a unique element $m \in \BMOA$ for which we have 
\begin{equation}
    \label{E:H1BMODualityEq}
\ell(f) = \lim_{r \to 1-} \int_\T f(r\zeta) \conj{m(r\zeta)} |d\zeta|
\end{equation} The operator norm of the functional $\ell$ is comparable to $\|m\|_{\BMOA} := \inf \|u\|_{\infty}$, this infimum extending over all bounded functions $u$ satisfying $m = P_+u$. We say that $\BMOA$ is the \textit{Cauchy dual} of $\hil^1$.

\subsection{Lotto-Sarason characterization} In \cite{lotto1993multipliers}, Lotto and Sarason characterized the membership $m \in \Mult{\hb}$ in terms of boundedness of compositions of Hankel operators and their adjoints. The formal adjoint of the operator $H_m$ is given by \[ H^*_m g = P_+ \conj{m} g\] where $g \in \conj{z\hil^2} \cap \conj{\hil^\infty}$, say. Then $H_m$ extends to a bounded operator $\hil^2 \to \conj{z\hil^2}$ if and only if $H^*_m$ extends to a bounded operator $\conj{z\hil^2} \to \hil^2$, and if this is the case, these operators are each others Hilbert space adjoints. The following result (see \cite[Theorem 2]{lotto1993multipliers} for a proof) characterizes multipliers for $\hb$:

\begin{thm} \thlabel{T:LottoSarason1}
    Let $b$ be a non-extreme point of the unit ball of $\hil^\infty$, and let $a$ be the Pythagorean mate of $b$. A function $m \in \hil^\infty$ is a member of $\Mult{\hb}$ if and only if the following three conditions hold.
    \begin{enumerate}[(i)]
        \item There exists $u \in \hil^2$ such that $m = \To_{\conj{a}}u$.
        \item The operator $H^*_{\conj{u}} H_{\conj{a}}$ is bounded on $\hil^2$.
        \item The operator $H^*_{\conj{u}} H_{\conj{b}}$ is bounded on $\hil^2$.        
    \end{enumerate}
\end{thm}

Note that the conditions $(ii)$ and $(iii)$ may be satisfied even without the operator $H^*_{\conj{u}}$ being bounded by itself. The question of when a composition of Hankel operators as above is bounded is rather delicate. \thref{T:LottoSarason1} gives us, however, a sufficient condition for membership of $m \in \Mult{\hb}$. Namely, if the solution $u$ to the equation $m = \To_{\conj{a}} u$ happens to satisfy $u \in \BMOA$, then the operator $H_{\conj{u}}$ is bounded by Nehari's theorem, and hence conditions $(ii)$ and $(iii)$ in \thref{T:LottoSarason1} are satisfied by virtue of $a,b \in \hil^\infty \subset \BMOA$. This simple consequence of \thref{T:LottoSarason1} will play an important role in our development. 

\begin{cor}
    \thlabel{C:LottoSarasonCor1}
    If $m \in \hil^\infty$ is of the form \[ m = \To_{\conj{a}} u\] for $u \in \BMOA$, then $m \in \Mult{\hb}$.
\end{cor}

Lotto and Sarason in \cite{lotto1993multipliers} used \thref{C:LottoSarasonCor1} to give new proofs of certain known statements regarding multipliers on $\hb$. They proved also a remarkable theorem characterizing when a multiplier $m \in \Mult{\hb}$ is simultaneously a multiplier for the family of spaces $\{ \hil(Ib) \}_I$, where $I$ is any inner function.

\begin{thm}
    \thlabel{T:LottoSarason2}
    Let $m = \To_{\conj{a}} u \in \Mult{\hb}$. The following two statements are equivalent. 
    \begin{enumerate}[(i)]
        \item We have \[ m \in \bigcap_{I} \Mult{\hil(Ib)},\] where the intersection is taken over all inner functions $I$.
        \item The Hankel operator $H_{\conj{u}b}: \hil^2 \to \conj{z\hil^2}$ is bounded. 
    \end{enumerate}
\end{thm}

The easily verified identity $H_{\conj{u}b} = H_{P_-\conj{u}b}$ and Nehari's theorem imply that the second condition in \thref{T:LottoSarason2} is equivalent to $\conj{P_- \conj{u}b}$ being a member of $\BMOA$. In fact, one can express the condition $(ii)$ intrinsically in terms of the space $\hb$. Indeed, condition $(ii)$ is equivalent to the mate $m_+$ of $m$ (recall the mate equation in \eqref{E:MateEquationIntro}) being contained in $\BMOA$, as follows from the proof of the following corollary.

\begin{cor} 
\thlabel{C:LottoSarasonCor2}
Let $\mathcal{F}$ be a family of symbols invariant under multiplication by inner functions: $b\in \mathcal{F}$ implies that $Ib \in \mathcal{F}$ for every inner function $I$. If $m$ is a universal multiplier for $\mathcal{F}$, then for any $b \in \mathcal{F}$, the mate $m_+$ of $m$ in $\hb$ is a member of $\BMOA$.   \end{cor}

\begin{proof}
Since $m$ is contained in the intersection of the algebras $\Mult{\hb}$ for $b \in \mathcal{F}$, in particular it is contained in the intersection of the algebras $\Mult{\hil(Ib)}$, where $I$ is an arbitrary inner function. By \thref{T:LottoSarason2}, the operator $H_{\conj{u}b}:\hil^2 \to \conj{z\hil^2}$ is bounded, where $u \in \hil^2$ satisfies $m = \To_{\conj{a}}u$. By Nehari's theorem, $P_-(\conj{u}b) \in \conj{\BMOA}$, the space of complex conjugates of functions in $\BMOA$. Relation \eqref{E:PplusPminusEq} implies that $P_+(\conj{b} u) = \To_{\conj{b}}u \in \BMOA$. Now, note that by the commutation relation stated in \eqref{E:ToeplitzHomRelation} below we have
    \[ \To_{\conj{a}} \To_{\conj{b}} u = \To_{\conj{b}}\To_{\conj{a}}u = \To_{\conj{b}}m.\] This identity and \eqref{E:MateEquationIntro} tells us that $\To_{\conj{b}}u$ is the mate of $m$ in $\hb$. That is, $m_+ = \To_{\conj{b}}u \in \BMOA$. 
\end{proof}

\section{Hankel operators and their continuity} \label{S:HankelSection}

Proofs of our main results will rely on characterization of the continuity of Hankel operators $H_{\conj{m}}$ acting between a space $X$ and $\conj{\BMOA}$. This section presents elementary material on Hankel operators between spaces of analytic functions, and conditions for their continuity. The main results are \thref{P:HankelBoundednessHpToBMOA} and \thref{P:HankelGevreyBoundedness}.

\subsection{Hankel operators with Lipschitz symbols}
We deal first with the case $X = \hil^p$ for $p \in (0,\infty)$. Note that the definition of the Hankel operator in \eqref{E:HankelOpDef} involving the projection operator $P_-$ in \eqref{E:PminusDef} presupposes the integrability of $\conj{m}f$ on $\T$. Therefore, in the case that $p < 1$, we may initially only define $H_{\conj{m}}$ on a dense subset of $\hil^p$. In any case, let us say that $H_{\conj{m}}$ is continuous as an operator from $\hil^p$ into $\conj{\BMOA}$ if there exists a constant $C > 0$ such that for every function $h \in \hil^\infty$ we have the estimate
\begin{equation}
    \label{E:HmBoundednessCond}
\| H_{\conj{m}}h\|_{\conj{\BMOA}} \leq C \|h\|_p, \end{equation} where we use the natural definition $\|g \|_{\conj{\BMOA}} := \| \conj{g}\|_{\BMOA}$ and where $\|\cdot\|_p$ was defined in \eqref{E:HpNorm}. If \eqref{E:HmBoundednessCond} holds, then $H_{\conj{m}}$
extends by continuity from the (dense) subset $\hil^\infty$ to $\hil^p$, and the extension is a continuous linear operator from $\hil^p$ into $\conj{\BMOA}$.

\begin{prop}
    \thlabel{P:HankelBoundednessHpToBMOA} Let $p \in (0,\infty)$ and $m \in \hil^\infty$. The Hankel operator $H_{\conj{m}}$ is continuous from $\hil^p$ into $\conj{\BMOA}$ if and only if $m \in \Lambda^a_{1/p}$.
\end{prop}

The proof is essentially the same as the classical proof of boundedness of the Hankel operator $H_{\conj{m}}:\hil^2 \to \conj{z\hil^2}$ being equivalent to $m \in \BMOA$. The pivotal point in the classical proof is the ability to factor a function in the predual to $\BMOA$, which is $\hil^1$, into a product of two functions in $\hil^2$. We follow this idea in our proof of \thref{P:HankelBoundednessHpToBMOA}.

In order to apply the reasoning in the last paragraph, we need to identify the Cauchy predual space of the analytic Lipschitz classes. This is done in the following result.

\begin{lem}[Duren-Romberg-Shields, Hardy-Littlewood] \thlabel{L:DRSDualityTheorem}
Let $q \in (0, 1)$ and $m \in \hil^\infty$. The following statements are equivalent.

\begin{enumerate}[(i)]
    \item The limit \[\ell_m(f) := \lim_{r \to 1-} \int_\T f(r\zeta) \conj{m(r\zeta)} |d\zeta|\] exists for every $f \in \hil^q$ and satisfies the bound \[ |\ell_m(f)| \leq C_m \|f\|_q\] for some constant $C_m > 0$.
    \item We have $m \in \Lambda^a_{1/q - 1}$.
\end{enumerate}    
\end{lem}

In other words, the Cauchy dual of $\hil^q$ is $\Lambda^a_{1/q - 1}$. Duren, Romberg and Shields give a proof in \cite{romberg1969linear}. The critical estimate in their proof is an inequality for integral means of analytic functions due to Hardy and Littlewood from \cite[page 412]{hardy1931fractionalintegrals}.

We need also an elementary factorization result.

\begin{lem} \thlabel{L:FactorizationHp}
Let the positive real numbers $p, q, s$ be related by \[ \frac{1}{q} = \frac{1}{s} + \frac{1}{p}.\]

\begin{enumerate}[(i)]
    \item Every function $f \in \hil^q$ can be factored as $f= g\cdot h$, with $g \in \hil^s$ and $h \in \hil^p$ in such a way that \[ \|f\|_q = \|g\|_s \cdot \|h\|_p.\]
    \item  Conversely, given $g \in \hil^s$ and $h \in \hil^p$, the product $f = gh$ satisfies 
    \[ \|f\|_q \leq \|g\|_s \cdot \|h\|_p\]    
\end{enumerate}
\end{lem}

Part $(i)$ follows readily from the inner-outer factorization of Hardy space functions, while part $(ii)$ follows from an application of Hölder's inequality.

\begin{proof}[Proof of \thref{P:HankelBoundednessHpToBMOA}]
Fix $p \in (0,\infty)$ and let $m \in \Lambda^a_{1/p}$. To show that $H_{\conj{m}}$ is continuous from $\hil^p$ into $\conj{\BMOA}$ it will suffice to show by the $\hil^1-\BMOA$ duality that we have an estimate of the form
\[ \Big\vert \int_\T g H_{\conj{m}}h |d\zeta| \Big\vert \leq C_m \|g\|_1 \|h\|_p\] for every pair of functions $g$ and $h$ in $\hil^\infty$. Note that the left-hand side in the inequality above vanishes if $g$ is a constant function, and so we may assume that $g(0) = 0$. Set $f = gh$. The expression inside the absolute value on the right-hand side above equals the inner product $\ip{P_- \conj{m}h}{\conj{g}}$ in $L^2(\T)$, and since $\conj{g} \in \conj{z\hil^2} = L^2(\T) \ominus \hil^2$, it equals \begin{align*}
    \int_\T g H_{\conj{m}}h |d\zeta| &= \ip{\conj{m}h}{\conj{g}} \\
    &= \int_\T \conj{m} gh |d\zeta| \\
    &=  \int_\T \conj{m} f |d\zeta|. 
\end{align*}
If $q$ satisfies $\frac{1}{q} = 1 + \frac{1}{p}$, then $\Lambda^a_{1/p} = \Lambda^a_{1/q-1}$, and by \thref{L:DRSDualityTheorem} and part $(ii)$ of \thref{L:FactorizationHp}, last expression is bounded in modulus by \[C_m \|f\|_q \leq C_m \|g\|_1 \|h\|_p.\] Thus $H_{\conj{m}}$ is continuous from $\hil^p$ into $\conj{\BMOA}$.

If conversely $H_{\conj{m}}$ is continuous from $\hil^p$ into $\conj{\BMOA}$, then to show that $m \in \Lambda^a_{1/p}$ it will suffice by \thref{L:DRSDualityTheorem} to show that \[ \Big\vert \int_\T \conj{f} m |d\zeta| \Big\vert \leq C \|f\|_q\] for some constant $C > 0$, where $q$ is as in the proof of the previous implication. If $f$ is a constant function, then the inequality is obvious, and so we may assume that $f(0) = 0$. Factor $f = gh$ according to part $(i)$ of \thref{L:FactorizationHp}, so that $g(0) = 0$, and $\|f\|_q = \|g\|_1 \cdot \|h\|_p$. Then $\conj{g} \in \conj{z\hil^2}$, and so 
\begin{align*}
    \int_\T \conj{f}m |d\zeta| &= \ip{\conj{gh}}{\conj{m}} \\
    &= \ip{\conj{g}}{H_{\conj{m}} h}\\
    &= \int_\T \conj{g} \conj{H_{\conj{m}}h} |d\zeta|
\end{align*}
Using the $\hil^1-\BMOA$ duality and the assumption of continuity of $H_{\conj{m}}$, the last expression is bounded in modulus by
\[ C \|g\|_1 \cdot \|H_{\conj{m}}h\|_{\conj{\BMOA}} \leq C' \|g\|_1 \cdot \|h\|_p = C'\|f\|_q.\] The proof is complete.
\end{proof}

\subsection{Hankel operators with Gevrey symbols} \label{S:HankelGevreySubsec}
The purpose now is to derive results similar to the ones above, but which apply to Hankel operators with Gevrey symbols defined in \eqref{E:GevreyEq}. The domain of the Hankel operator $H_{\conj{m}}$ will be the Privalov class $\N^q$ for $q > 1$, or the Smirnov class $\N^+$. Topologies on these spaces will be induced by the translation-invariant metrics 
\begin{equation}
\label{E:PrivalovMetricDef}
\|f-g\|_{\N^q} := \int_{\T} \big(\log(1+|f-g|)\big)^q |d\zeta|, \quad f,g \in \N^q,
\end{equation}
and
\begin{equation}
\label{E:SmirnovMetricDef} \|f-g\|_{\N^+} := \int_{\T} \log(1+|f-g|) |d\zeta|, \quad f,g \in \N^+.
\end{equation}

With these definitions, $\N^q$ and $\N^+$ become so-called $F$-spaces (see \cite[Section 1.8]{rudin1991functional}). A topological vector space $X$ is said to be an $F$-space if scalar multiplication and addition are continuous in $X$, and the topology of $X$ is induced by a complete and translation-invariant metric. Whenever we mention the space $\N^q$ or $\N^+$, it is understood that they are equipped with the respective metric topologies. 

We say that the Hankel operator $H_{\conj{m}}$ is continuous from $\N^q$ (or $\N^+$) into $\conj{\BMOA}$ if $H_{\conj{m}}$ defined in \eqref{E:HankelOpDef} has a continuous extension to an operator from $\N^q$ (or $\N^+)$ into $\conj{\BMOA}$.

\begin{prop}  \thlabel{P:HankelGevreyBoundedness}
\quad
\begin{enumerate}[(i)] 
\item For $q > 1$, the Hankel operator $H_{\conj{m}}: \N^q \to \conj{\BMOA}$ is continuous if and only if $m \in \mathcal{G}_{1/(1+q)}$.
\item The Hankel operator $H_{\conj{m}}: \N^+ \to \conj{\BMOA}$ is continuous if and only if $m \in \mathcal{G}_{1/2}$.
\end{enumerate}
\end{prop}

In fact, characterization of the continuity of Hankel operators in this case is rather insensitive to change of the range space. In $(i)$, we may replace $\conj{\BMOA}$ by $\conj{\N^q}$ itself, or any of the spaces $\conj{\hil^p}$, even for $p = \infty$, without changing the conclusion. So, for instance, membership of $m$ in the corresponding Gevrey class is equivalent to continuity of the operator    $H_{\conj{m}}:\N^q \to \conj{\hil^p}$. Similar statement holds for the operators $H_{\conj{m}}$ on $\N^+$ also. These facts can be deduced from the proofs below, but they will not be used in the sequel.

To prove \thref{P:HankelGevreyBoundedness}, we will need counterparts of \thref{L:DRSDualityTheorem}.

\begin{lem} \thlabel{L:PrivalovDuality} Let $m \in \hil^\infty$ and set \[\ell_m(f) := \int_\T f(\zeta) \conj{m(\zeta)} |d\zeta| = 2\pi \sum_{k=0}^\infty \widehat{f}(k)\conj{\widehat{m}(k)}, \quad f \in \hil^\infty.\]
\begin{enumerate}[(i)]
    \item For $p > 1$, the mapping $\ell_m$ extends to a continuous linear functional on $\N^q$ if and only if $m \in \mathcal{G}_{1/(1+q)}$.
    \item The mapping $\ell_m$ extends to a continuous linear functional on $\N^+$ if and only if $m \in \mathcal{G}_{1/2}$.
\end{enumerate}    
\end{lem}

Part $(ii)$ is a well known result of Yanagihara from \cite{yanagihara1973multipliers}, while part $(i)$ is due to Me{\v{s}}trovi{\'c} and Pavi{\'c}evi{\'c} in \cite{mevstrovic2003topologies}, who gave a proof using earlier results of Stoll from \cite{stoll1977mean}. Moreover, in \cite{stoll1977mean} and \cite{yanagihara1973multipliers} the following elementary Taylor coefficient estimates are contained.

\begin{lem} \thlabel{L:PrivalovCoefficients}
\quad 
\begin{enumerate}[(i)]
\item For any $q > 1$ and $f \in \N^q$, we have the estimate \[ |\widehat{f}(n)| = \exp\big( o(n^{1/(1+q)})\big), \quad n \geq 0.\] 
\item For $f \in \N^+$, we have the estimate \[ |\widehat{f}(n)| = \exp\big( o(n^{1/2})\big), \quad n \geq 0.\]
\end{enumerate}
\end{lem}

\begin{proof}[Proof of \thref{P:HankelGevreyBoundedness}]
Parts $(i)$ and $(ii)$ have the same proof. For the moment, let us assume that  $f \in \hil^\infty$ and note that \eqref{E:HankelOpDef} implies the Fourier series representation 
\begin{equation}
\label{E:HankelFourierSeriesRepr}
\widehat{H_{\conj{m}}f}(n) = 2\pi \sum_{k=0}^\infty \widehat{f}(k)\conj{\widehat{m}(k-n)} = \int_\T f(\zeta)\zeta^{|n|}\conj{m(\zeta)}  |d\zeta|, \quad n \leq -1\end{equation} and $\widehat{H_{\conj{m}}f}(n)=0$ for $n \geq 0$.  If $H_{\conj{m}}$ is a continuous operator from $\N^q$ (or $\N^+$) to a space on which the mappings $g \mapsto \widehat{g}(n)$ are continuous (in particular, this applies to $\conj{\BMOA}$), then
$f \mapsto \widehat{H_{\conj{m}}f}(-1)$ is a continuous linear functional on $\N^q$ (or $\N^+$), and from \eqref{E:HankelFourierSeriesRepr} and \thref{L:PrivalovDuality} we deduce readily that $m$ lies in the corresponding Gevrey class. 

We will prove the converse statement for case $(ii)$ , the proof for the case $(i)$ being analogous. Since we are assuming that $m \in \mathcal{G}_{1/2}$, the  series $\sum_{k=0}^\infty \widehat{f}(k)\conj{\widehat{m}(k-n)}$ converges absolutely for \textit{every} $f \in \N^+$ and $n \leq -1$. In fact, we have
\begin{equation}
\label{E:Claim1}
\sum_{k=0}^\infty \big\vert \widehat{f}(k)\conj{\widehat{m}(k-n)} \big\vert \leq A\exp\big( -d |n|^{1/2} \big) \end{equation} for some constants $A > 0$ and $d > 0$ depending only on $f$ and $m$. Accepting for a moment the claim, note that the linear operator \[f \mapsto 2\pi \sum_{n \leq -1} \Big(\sum_{k=0}^\infty \widehat{f}(k)\conj{\widehat{m}(k-n)}\Big)\conj{z^n} \] maps $\N^+$ into $\conj{\mathcal{G}_{1/2}} \subset \conj{\BMOA}$, coincides for $f \in \hil^\infty$ with our earlier definition of $H_{\conj{m}}$ by \eqref{E:HankelFourierSeriesRepr}, and it is easily seen to be closed (since the linear functionals $f \mapsto \sum_{k=0}^\infty \widehat{f}(k)\conj{\widehat{m}(k-n)}$ are continuous on $\N^+$ for each $n$, by virtue of $m \in \mathcal{G}_{1/2}$ and \thref{L:PrivalovDuality}). Since $\N^+$ and $\conj{\BMOA}$ are $F$-spaces, the usual formulation of the closed graph theorem applies (see \cite[Section 2.15]{rudin1991functional}), and we conclude that $H_{\conj{m}}: \N^+ \to \conj{\BMOA}$ is continuous.

It remains to verify the claim \eqref{E:Claim1}. Since $m \in \mathcal{G}_{1/2}$, we have that  $|\widehat{m}(k)| \leq B \exp (-3dk^{1/2})$ for some $B > 0$, $d > 0$ and every integer $k \geq 0$. By part $(ii)$ of \thref{L:PrivalovCoefficients}, there exists a positive integer $K$ such that $|\widehat{f}(k)| \leq \exp( dk^{1/2})$ for $k \geq K$. Then, for integers $n \leq -1$, we have 
\begin{align*}
\big\vert \widehat{f}(k)\conj{\widehat{m}(k-n)} \big\vert &\leq B\exp(-2d(k+|n|)^{1/2}) \\ &\leq B\exp(-d|n|^{1/2})\exp(-dk^{1/2})
\end{align*} for $k \geq K$. If $D = \max \{ |\widehat{f}(k)| : 0 \leq k < K \}$, we may now estimate 
\begin{align*}
\sum_{k=0}^\infty \big\vert \widehat{f}(k)\conj{\widehat{m}(k-n)} \big\vert &\leq DK\exp(-d|n|^{1/2}) + B\exp(-d|n|^{1/2})\sum_{k=K}^\infty \exp(-dk^{1/2}) \\
&\leq A \exp(-d|n|^{1/2}).
\end{align*} We have verified \eqref{E:Claim1}, and so the proof is complete (in the proof of case $(i)$, the above estimates are all valid with the exponent $1/2$ replaced by $1/(1+q)$).
\end{proof}

\section{Pythagorean factorizations}\label{S:PythagoreanSection}

To each symbol $b$ we have associated the quotient $h = b/a \in \N^+$. This process can be reversed, and so every $h \in \N^+$ can be uniquely expressed as a quotient of Pythagorean mates. In this section, we prove a stability result for this factorization.

\subsection{Pythagorean factorization and their stability}

Let $h \in \N^+$ and let $I$ be the inner factor of $h$. Set $a$ to be the unique outer function satisfying $a(0) > 0$ and \begin{equation}
    \label{E:AbdryValH} |a|^2 = \frac{1}{|h|^2 + 1}
\end{equation} almost everywhere on $\T$. Then there exists a unimodular scalar $c$ and a unique outer function $b_o$ which satisfies $b_o(0) > 0$, \[ |b_o|^2 = \frac{|h|^2}{|h|^2 + 1}\] and \[h(z) = \frac{cI(z)b_o(z)}{a(z)} = \frac{b(z)}{a(z)}, \quad z \in \D.\] We say that $h = b/a$ is the \textit{Pythagorean factorization} of $h$. Clearly $b$ and $a$ are Pythagorean mates, in the sense that the equality $|b|^2 + |a|^2 = 1$ holds almost everywhere on $\T$. 

We will need a convergence result for mates in the factorizations.

\begin{prop}
    \thlabel{P:PythFactorStabProp} Assume that the sequence of functions $\{h_n\}_n$ converges to the function $h$ in the metric topology on $\N^+$. If $h_n = b_n/a_n$ and $h = b/a$ are the corresponding Pythagorean factorizations, then there exists a subsequence $\{n_k\}_k$ such that
    \[ \lim_{k \to \infty} a_{n_k}(\zeta) = a(\zeta)\] and
    \[ \lim_{k \to \infty} b_{n_k}(\zeta) = b(\zeta)\] for almost every $\zeta \in \T$.
\end{prop}

We prove \thref{P:PythFactorStabProp} after a brief review of basic properties of outer functions.

\subsection{Pointwise convergence of outer functions on the boundary}
Recall that if $a$ is an outer function satisfying $a(0) > 0$, and we set $g(\zeta) = \log |a(\zeta)|$, then we have the representation formula \[ \log a(z) = \frac{1}{2\pi}\int_\T \frac{\zeta + z}{\zeta - z} g(\zeta) |d\zeta|  = \widehat{g}(0) + 2 \sum_{n \geq 1} \widehat{g}(n) z^n = 2 P_+g(z) - \widehat{g}(0), \quad z \in \D.\]
For any $p \in (0,1)$ the operator $P_+$ is continuous from $L^1(\T)$ into $\hil^p$ (see, for instance, \cite[Theorem 2.1.10]{cauchytransform}). Thus if $g_n(\zeta) = \log |a_n(\zeta)|$ and $g_n \to g$ in $L^1(\T)$, then after passing to a subsequence we can ensure the pointwise convergence 
\[ \lim_n \log a_n(\zeta) = \log a(\zeta)\] for almost every $\zeta \in \T$. After exponentiating this translates into 
\[ \lim_n a_n(\zeta) = a(\zeta)\] for almost every $\zeta \in \T$. The proof of \thref{P:PythFactorStabProp} is thus reduced to showing that the convergence $h_n = b_n/a_n \to h = b/a$ in $\N^+$ implies that $\log |a_{n_k}| \to \log |a|$ in $L^1(\T)$ along some subsequence $\{n_k\}_k$. For if this is the case, then since $h_n \to h$ in $\N^+$ implies by \eqref{E:SmirnovMetricDef} that
\begin{equation}
\label{E:hnhNConvergenceEq}
\lim_{n \to \infty} \int_\T \log(1 + |h_n-h|) |d\zeta| = 0,
\end{equation} 
basic measure theory ensures that the subsequence can be refined to ensure almost everywhere pointwise convergence $h_{n_k}(\zeta) \to h(\zeta)$. By the above reasoning we can ensure also the convergence $a_{n_k}(\zeta) \to a(\zeta)$ almost everywhere, and then it follows that $b_{n_k}(\zeta) \to b(\zeta)$ almost everywhere.

\subsection{Proof of the stability result}

The sought-after result now follows in a standard way from uniform integrability (see \cite[page 133]{rudinrealcomplex}, for instance). Recall that a sequence $\{g_n\}_n$ of functions in $L^1(\T)$ is said to be uniformly integrable if
\[ \lim_{\delta \to 0} \sup_{E : |E| \leq \delta} \int_E |g_n| |d\zeta| = 0\] holds uniformly in $n$. An elementary argument shows that if $g_n \to g$ almost everywhere on $\T$, then uniform integrability of the sequence $\{g_n\}_n$ is equivalent to convergence $g_n \to g$ in $L^1(\T)$.

\begin{proof}[Proof of \thref{P:PythFactorStabProp}] 
According to the above discussion, it suffices for us to show that $\log |a_n| \to \log |a|$ in $L^1(\T)$. Since $h_n \to h$ in $\N^+$, after passing to a subsequence, we may suppose that $h_n(\zeta) \to h(\zeta)$ almost everywhere on $\T$. Recalling \eqref{E:AbdryValH}, this implies that
\[ \lim_n |a_n(\zeta)|^2 = \lim_n \frac{1}{|h_n(\zeta)|^2 + 1} = \frac{1}{|h(\zeta)|^2 + 1} = |a(\zeta)|^2 \] almost everywhere on $\T$. Taking logarithms, we obtain $\log |a_n(\zeta)| \to \log |a(\zeta)|$ almost everywhere on $\T$. 

Note that $1/|a|^2 = |h_n|^2 + 1$ on $\T$ and recall that $\sqrt{x+1} \leq \sqrt{x} + 1$ for $x \geq 0$. Use also that the logarithm is increasing to see that on $\T$ we have 

\begin{align*}
\big\vert \log |a_n|\big\vert   &= \log \big( \sqrt{|h_n|^2 + 1}\big) \\
&\leq \log\big( |h_n| + 1 \big) \\
&\leq \log \big(1 + |h_n-h|\big) + \log \big(1 + |h|\big).
\end{align*}
Since \eqref{E:hnhNConvergenceEq} holds, the sequence $\{\log(1+|h_n-h|)\}_n$ converges in $L^1(\T)$ (to $0$), and so is uniformly integrable. But then the above inequalities show that $\{ \log |a_n| \}_n$ is uniformly integrable on $\T$, and since $\log |a_n| \to \log |a|$ almost everywhere on $\T$, we have also $\log |a_n| \to \log |a|$ in $L^1(\T)$. By the initial remarks, the proof is complete.
\end{proof}

\section{Proof of the main theorems} \label{S:MainTheoremProofSection}

\subsection{A general criterion} \label{S:GeneralCriterionSection}

\thref{T:MainTheorem1} and \thref{T:MainTheorem2} will follow from our earlier developments as corollaries of a general statement which we shall now state, and prove next. 

In this section, we let $X$ be a topological space of analytic functions on $\D$. The properties of $X$ which we shall need in our proof are as follows.

\begin{enumerate}[(i)]
\item $X$ is an $F$-space.
\item $X$ is continuously contained in the Smirnov class $\N^+$.
\item $\hil^\infty$ is dense in $X$. 
\item The multiplication operator $f \mapsto \varphi f$ is continuous on $X$ for each $\varphi \in \hil^\infty$.
\item If $h \in X$ has Pythagorean factorization $h = b/a$, then $1/a \in X$.
\end{enumerate}

It is known that the spaces appearing in theorems stated in the Introduction, namely $\hil^p$, $\N^q$ and $\N^+$, all satisfy the above five conditions. See, for instance, \cite{garnett}, \cite{PrivalovBook} and \cite{stoll1977mean}. The fifth property is easily seen by recalling that $1/|a|^2 = |h|^2 + 1$ on $\T$.

Our general result will apply to any space $X$ satisfying the above five conditions. To such $X$ we associate two sets of functions. As before, we let $\mathcal{F}(X)$ consist of non-extreme symbols $b$ with Pythagorean mate $a$ for which $b/a$ is a member of $X$, and we let $H(X, \conj{\BMOA})$ consist of those analytic functions $m \in \hil^\infty$ for which the densely defined Hankel operator $H_{\conj{m}}$ in \eqref{E:HankelOpDef} extends to a continuous mapping from $X$ into $\conj{\BMOA}$. 

\begin{prop} \thlabel{P:MainProp}
Let $X$ be a topological space of analytic functions on $\D$ satisfying properties $(i)-(v)$ stated above. Then the universal multipliers for $\mathcal{F}(X)$ coincide with the symbols of continuous Hankel operators $H_{\conj{m}}: X \to \conj{\BMOA}$. That is, we have the equality
\[ \bigcap_{b \in \mathcal{F}(X)} \Mult{\hb} = H(X, \conj{\BMOA}). \]
\end{prop}

\thref{T:MainTheorem1}, \thref{T:MainTheorem2} and the Davis-McCarthy Theorem follow immediately from \thref{P:MainProp} and the characterization of continuous Hankel operators in \thref{P:HankelBoundednessHpToBMOA} and \thref{P:HankelGevreyBoundedness}.

We proceed with the proof of \thref{P:MainProp}.

\subsection{Toeplitz operators} \label{S:ToeplitzSubsec} For notational convenience, we shall use the Toeplitz operators \begin{equation}
    \label{E:ToeplitzDef}
    \mathcal{T}_g : \hil^2 \to \hil^2, \quad f \mapsto P_+(\conj{g}f).
\end{equation} Here $g \in L^\infty(\T)$ is the symbol of the operator. The following commutation relation is readily established and will be used frequently: 
\begin{equation}
    \label{E:ToeplitzHomRelation}
\To_{\conj{g_1g_2}} = \To_{\conj{g_1}}\To_{\conj{g_2}} = \To_{\conj{g_2}}\To_{\conj{g_1}}, \quad g_1, g_2 \in \hil^\infty. \end{equation}

If $g \in \hil^\infty$, then the operator $T_{\conj{g}}$ is bounded from $\BMOA$ into itself. To see this, we may simply observe that $\To_{\conj{g}}$ is the Banach space adjoint of the operator of multiplication by $g$ on $\hil^1$, with respect to the $\hil^1- \BMOA$ duality pairing in \eqref{E:H1BMODualityEq}. Moreover, if $g$ is an outer function, then the operator $\To_{\conj{g}}$ is injective on $\hil^2$.

\subsection{From Hankel continuity to universal multiplier} \label{S:ProofSecPart1}
We will first prove that if $m \in H(X, \conj{\BMOA})$, then $m$ is a universal multiplier for the family $\mathcal{F}(X)$. Let $b \in \mathcal{F}(X)$, so that $b/a \in X$. By properties $(iii)$ and $(v)$ in Section \ref{S:GeneralCriterionSection}, there exists a sequence $\{h_n\}_n$ of functions in $\hil^\infty$ such that $h_n \to 1/a$ in the topology of $X$. By property $(iv)$, we have $ah_n \to 1$ in $X$. Now, $H_{\conj{m}}(a^{-1}) \in \conj{\BMOA}$, and this function vanishes at $z = 0$. Say, \[H_{\conj{m}}(a^{-1}) = P_-(\conj{m} a^{-1}) = \conj{zu},\] where $u \in \BMOA$. Using the identity \eqref{E:PplusPminusEq} and continuity of $H_{\conj{m}}$ we obtain 
\begin{equation}
    zu = \lim_{n} \conj{ H_{\conj{m}} h_n} = \lim_n P_+ (\conj{h_n}m) + c_n = \lim_n \To_{\conj{h_m}}m + c_n,
\end{equation} with convergence in the sense of the norm on $\BMOA$ and where $c_n$ are constants. Applying the operator $\To_{\conj{z}}$ (here $z$ denotes the identity function on $\T$) and using the relation \eqref{E:ToeplitzHomRelation}, we arrive at \[ u = \lim_n \To_{\conj{z h_n}} m.\]  Applying also the Toeplitz operator $\To_{\conj{a}}$, which, as mentioned above, is continuous on $\BMOA$, and using \eqref{E:PplusPminusEq}, we obtain
\begin{align*}
    \To_{\conj{a}} u & = \lim_n \To_{\conj{zah_n}} m\\
    &= \lim_n P_+( \conj{az h_n} m) \\
    &= \lim_n \conj{H_{\conj{m}}(azh_n)} + \lim_n c'_n
\end{align*} where $c'_n$ are constants. Since $ah_n \to 1$ in $X$, the continuity of the Hankel operator $H_{\conj{m}}$ implies that \[H_{\conj{m}}(azh_n) \to H_{\conj{m}}z = \frac{\conj{m - m(0)}}{\conj{z}} - \conj{m'(0)}\] in $\conj{\BMOA}$. We must then also have that $\lim_n c'_n = c'$ for some constant $c'$. Finally, we obtain that 
\[ \To_{\conj{a}}u = \frac{m - m(0)}{z} - m'(0) + c'.\] Since $u \in \BMOA$, we obtain by \thref{C:LottoSarasonCor1} that $\frac{m - m(0)}{z} - m'(0) + c' \in \Mult{\hb}$. It follows that $m \in \Mult{\hb}$.

We have therefore proved that the right-hand side is included in the left-hand side in the asserted equality in the statement of \thref{P:MainProp}.

\subsection{From universal multiplier to Hankel continuity}

It takes considerably more effort to prove that the universal multiplier property of $m$ for the family $\mathcal{F}(X)$ implies that $m \in H(X,\conj{\BMOA})$. We start with a few observations.

If $m$ is a universal multiplier for the family $\mathcal{F}(X)$, then for any $h = b/a \in X$ in particular we have $m \in \hb$, and so a unique mate $m_+(h) \in \hil^2$ exists which satisfies the operator equation 
\begin{equation}
    \label{E:MateEquationHb}
    \To_{\conj{b}}m= \To_{\conj{a}}m_+(h).
\end{equation} 

The use of notation $m_+(h)$ in favor of the more natural $m_+(b)$ is a conscious choice, as we will soon see that $h \mapsto \conj{m_+(h)}$ is a linear function. Let us note that our assumptions force $m_+(h)$ to lie in $\BMOA$. By property $(iv)$ in Section \ref{S:GeneralCriterionSection}, if $b/a \in X$ and $I$ is an inner function, then $Ib/a \in X$. But then $Ib \in \mathcal{F}(X)$, since $a$ is the Pythagorean mate of $Ib$. Hence $\mathcal{F}(X)$ satisfies the hypothesis of \thref{C:LottoSarasonCor2}. Thus the mapping
\begin{equation} \label{E:TOpeatorDef}
\conj{m_+} : X \to \conj{\BMOA}, \quad h = b/a \, \mapsto \, \conj{m_+(h)} 
\end{equation} is well-defined. Our task will be to show that $\conj{m_+}$ is linear and continuous. At the end of our development, we will note that $\conj{m_+}$ is a rank one perturbation of the operator $H_{\conj{m}}: X \to \conj{\BMOA}$.

We start by establishing linearity.

\begin{lem}
    \thlabel{L:TLinearity}
    For $h_1, h_2 \in X$ and a scalar $\lambda \in \mathbb{C}$, we have \[\conj{m_+(\lambda h_1 + h_2)} = \lambda \conj{m_+(h_1)} + \conj{m_+(h_2)}.\]
\end{lem}

\begin{proof} Let $h = \lambda h_1 + h_2$ and consider the Pythagorean factorizations 
\[ h_1 = \frac{b_1}{a_1}, \quad h_2 = \frac{b_2}{a_2}, \quad h = \frac{b}{a}.\] We have the identity
\[a_1a_2b = \lambda aa_2b_1 + aa_1b_2,\] and so
\begin{equation}
    \label{E:TLinearyProofEq}
    \To_{\conj{a_1a_2b}}m = \conj{\lambda} \To_{\conj{aa_2b_1}}m + \To_{\conj{aa_1b_2}}m.
\end{equation}
The relation \eqref{E:ToeplitzHomRelation} gives \[ \To_{\conj{a_1a_2b}} m = \To_{\conj{a_1a_2}}\To_{\conj{b}}m = \To_{\conj{a_1a_2}}\To_{\conj{a}}m_+(h) = \To_{\conj{aa_1a_2}}m_+(h).\] Similarly \[ \To_{\conj{aa_2b_1}}m = \To_{\conj{aa_1a_2}}m_+(h_1)\] and \[ \To_{\conj{aa_1b_2}}m = \To_{\conj{aa_1a_2}}m_+(h_2).\] Inputting these equalities into the relation \eqref{E:TLinearyProofEq} and rearranging, we obtain
\[ \To_{\conj{aa_1a_2}}\big(m_+(h) - \conj{\lambda} m_+(h_1) - m_+(h_2)\big) = 0.\] Because $aa_1a_2$ is an outer function, the above equality implies that $m_+(h) - \conj{\lambda} m_+(h_1) - m_+(h_2) = 0$, and so the proof is complete.
\end{proof}

We treat continuity next.

\begin{lem}
    \thlabel{L:TContinuity}
    The linear operator $\conj{m_+}: X \to \conj{\BMOA}$ is continuous.
\end{lem}

\begin{proof}
Since $X$ is an $F$-space by property $(i)$ in Section \ref{S:GeneralCriterionSection}, the usual formulation of the closed graph theorem is applicable to our situation (see \cite[Section 2.15]{rudin1991functional}). Hence the operator $\conj{m_+}$ will be continuous if we can verify the validity of the implication
\[ \begin{cases}
    h_n \to h & \text{ in } X\\ \conj{m_+(h_n)} \to \conj{g} & \text{ in } \conj{\BMOA}
\end{cases} \quad \Rightarrow \quad \conj{m_+(h)} = \conj{g}.\]

Above we assume that $g \in \BMOA$, so that $m_+(h_n) \to g$ in the norm of $\BMOA$ (recall that we have set $\|g\|_{\BMOA} := \|\conj{g}\|_{\conj{\BMOA}}$). Let the two above hypotheses be satisfied. Since $X$ is continuously contained in $\N^+$ by property $(ii)$ in Section \ref{S:GeneralCriterionSection}, by \thref{P:PythFactorStabProp} and by passing to a subsequence we may assume that we have the pointwise convergence $b_n \to b$ and $a_n \to a$ almost everywhere on $\T$, where $h_n = b_n/a_n$ and $h = b/a$ are the corresponding Pythagorean factorizations. Since the mapping $f \mapsto \To_{\conj{a}} f$ is continuous on $\BMOA$, we have
\begin{align}
    \label{E:ContProofEq1}
    \To_{\conj{a}}g &= \lim_n \To_{\conj{a}}m_+(h_n) \\
    &= \lim_n \To_{\conj{a-a_n}} m_+(h_n) + \To_{\conj{a_n}} m_+(h_n) \nonumber \\
    &= \lim_n \To_{\conj{a-a_n}}m_+(h_n) + \To_{\conj{b_n}}m \nonumber
\end{align} 
The convergence above holds in the norm of $\BMOA$, and so in particular in the sense of pointwise convergence on $\D$. Now, \begin{equation}
    \label{E:ContProofEq2}
    \To_{\conj{b_n}}m \to \To_{\conj{b}}m
\end{equation} in (say)  $\hil^2$, and so pointwise on $\D$, since $b_n \to b$ almost everywhere on $\T$. Since $m_+(h_n) \to g$ in $\BMOA$, the same is true in the space $\hil^4$. Therefore, by contractivity of the projection $P_+$ on $L^2(\T)$ and Hölder's inequality, we may estimate
\begin{align*}
    \| \To_{\conj{a-a_n}} m_+(h_n) \|_2^2 &= \|P_+ [\conj{(a - a_n)} m_+(h_n)] \|_2^2 \\ &\leq \| \conj{(a_n - a)}m_+(h_n)\|_2^2 \\ &\leq \| a_n - a\|_4^2 \cdot \|m_+(h_n)\|_4^2.
\end{align*} Since $a_n \to a$ almost everywhere on $\T$, we obtain $\|\To_{\conj{a-a_n}} m_+(h_n) \|_2 \to 0$, and from \eqref{E:ContProofEq1} and \eqref{E:ContProofEq2} we conclude that 
\[ \To_{\conj{a}}g(z) = \To_{\conj{b}}m(z), \quad z \in \D.\] According to \eqref{E:MateEquationHb}, this means that $g = m_+(h)$. We have thus verified that $\conj{m_+}: X \to \conj{\BMOA}$ is a closed operator. By the closed graph theorem this operator is also continuous.
\end{proof}

\begin{proof}[Completion of the proof of \thref{P:MainProp}]
Recalling the result of Section \ref{S:ProofSecPart1}, what remains to be established is that any universal multiplier $m$ for $\mathcal{F}(X)$ induces a continuous Hankel operator $H_{\conj{m}}:X \to \conj{\BMOA}$. By \thref{L:TContinuity}, for every $h =  b/a \in X$, the equation \eqref{E:MateEquationHb} has a unique solution $m_+(h)$ which depends continuously on $h \in X$. Let us suppose that $h = b/a \in \hil^\infty$. Then $1/a \in \hil^\infty$, and we may apply the Toeplitz operator $\To_{1/\conj{a}}$ and relation \eqref{E:ToeplitzHomRelation} to the mate equation in \eqref{E:MateEquationHb} to obtain
\[ \To_{\conj{h}}m = m_+(h).\] 
Since $\To_{\conj{h}}m = P_+\conj{h}m$, the relation \eqref{E:PplusPminusEq} implies that 
\begin{equation} \label{E:ToepliHankelConstEq} \conj{m_+(h)} = \conj{\To_{\conj{h}}m} = H_{\conj{m}}h + c(h) \end{equation} for a constant $c(h)$ which satisfies \[ |c(h)| = |\widehat{\conj{m}h}(0)|.\] In other words, $c(h)$ is the value at $z=0$ of the function $m_+(h)$. Since evaluations are continuous on $\BMOA$, it follows that $h \mapsto m_+(h)(0)$ is a continuous linear functional defined on all of $X$.  For $h \in \hil^\infty$, we have from \eqref{E:ToepliHankelConstEq} the equality \[H_{\conj{m}}h = \conj{m_+(h)} - m_+(h)(0).\] Since in the right-hand side of this equality appears a continuous linear mapping $X \to \conj{\BMOA}$, the left-hand side extends to a continuous linear operator $X \to \conj{\BMOA}$. That is, $m \in H(X, \conj{\BMOA})$.
\end{proof}

\section{Further comments and remarks}\label{sectfin}

We end the article with a list of a few unresolved matters and ideas for further research.

\subsection{Range spaces of coanalytic Toeplitz operators}
One consequence of following the presented new approach to universal multipliers is that we never needed to characterize the corresponding intersection of co-analytic Toeplitz operator ranges, which is a necessary step in the Davis-McCarthy approach. Denoting by $M(\conj{\varphi})$ the range of the co-analytic Toeplitz operator $\To_{\conj{\varphi}}:\hil^2 \to \hil^2$, a consequence of results from \cite{davis1991multipliers} is that we have 
\[ \mathcal{G}_{1/2} = \bigcap_{\varphi \in \hil^\infty\setminus \{0\}} M(\conj{\varphi}) = \bigcap_{\varphi \in \hil^\infty\setminus \{0\}} \Mult{M(\conj{\varphi})} = \bigcap_{b \in \mathcal{F}(\N^+)} \Mult{\hb}.\] Naively, one may expect from our main result that we should have a similar equality, with $\mathcal{G}$ replaced by $\Lambda^a_{1/p}$, and the intersections corresponding to spaces $M(\conj{\varphi})$ restricted to those (outer, bounded) symbols $\varphi$ for which $1/\varphi \in \hil^p$, namely
\[ \Lambda^a_{1/p} = \bigcap_{\varphi : 1/\varphi \in \hil^p } M(\conj{\varphi}) = \bigcap_{\varphi : 1/\varphi \in \hil^p} \Mult{M(\conj{\varphi})}.\]
That this is not the case in general is seen by setting $p = 2$. Then, for every $m \in \hil^\infty$  and $1/\varphi \in \hil^2$, we have $m = \To_{\conj{\varphi}}g$ where $g = P_+(m/\conj{\varphi})$, so $\hil^\infty \subseteq \bigcap_{1/\varphi \in \hil^2} M(\conj{\varphi})$. However, the equality 
\[ \Lambda^a_{1/p} = \bigcap_{\varphi: 1/\varphi \in \hil^p} \Mult{M(\conj{\varphi})}\] is not ruled out by this argument. We conjecture that this equality holds.

\subsection{Other families of symbols}
In this article, we focused on what we considered natural classes of symbols $\mathcal{F}(\hil^p)$ and $\mathcal{F}(\N^q)$, but similar questions may of course be explored about universal multipliers for families $\mathcal{F}(X)$ corresponding to any other reasonable space $X$. In relation to this, recall that we have characterized the multipliers of symbol families $\mathcal{F}(\hil^p)$ for finite $p$, and we note also that the case $b/a \in \hil^\infty$ corresponds to the equality $\hb = \hil^2$, with equivalence of norms. So $\Mult{\hb} = \hil^\infty$ if $b/a \in \hil^\infty$. We have not been able to compute the universal multipliers for the intermediate case $b \in \mathcal{F}(\BMOA)$. Our \thref{P:MainProp} does not apply in this case, since $\BMOA$ does not satisfy the hypotheses of that proposition. What are the universal multipliers in this case?

\subsection{Relating containment of a space within the multiplier algebra} For $p>2$, the authors own previous article \cite{MalmanSeco} allows for a different way to express our main theorem: there we showed that $b/a \in \hil^{\frac{2p}{p-2}}$ is equivalent with $\hil^p \subset \hil(b)$ (see Theorem A in \cite{MalmanSeco}). This makes plausible a description of universal multipliers for classes $\mathcal{F}$ described in terms of containing other fixed spaces of analytic functions. The case of the Dirichlet space promises to offer some resistance.

\subsection*{Acknowledgements}
The authors were funded through grant PID2023-149061NA-I00 by the Generaci\'on de Conocimiento programme and grant RYC2021-034744-I by the Ram\'on y Cajal programme from Agencia Estatal de Investigaci\'on (Spanish Ministry of Science, Innovation and Universities).


\normalsize

\bibliographystyle{plain}
\bibliography{mybib}

\end{document}